\documentclass[11pt]{article}
\usepackage{graphicx}
\usepackage{amsmath,amsthm,amssymb}
\usepackage{xypic}
\usepackage{mathabx}
\usepackage{bbm}
\usepackage{color}
\newtheorem*{theorem*}{Theorem}
\newtheorem*{conjecture*}{Conjecture}
\newtheorem{theorem}{Theorem}[section]
\newtheorem{proposition}[theorem]{Proposition}
\newtheorem{definition}[theorem]{Definition}
\newtheorem{example}[theorem]{Example}
\newtheorem{conjecture}[theorem]{Conjecture}
\newtheorem{remark}[theorem]{Remark}
\newtheorem{corollary}[theorem]{Corollary}

\newtheorem{principle}[theorem]{Principle}

\newcommand\blfootnote[1]{
  \begingroup
  \renewcommand\thefootnote{}\footnote{#1}
  \addtocounter{footnote}{-1}
  \endgroup
}

\begin{document}

\title{Higher Lawvere theories}
\author{John D. Berman}
\maketitle

\begin{abstract}\blfootnote{The author was supported by an NSF Postdoctoral Fellowship under grant 1803089.}
We survey Lawvere theories at the level of $\infty$-categories, as an alternative framework for higher algebra (rather than $\infty$-operads). From a pedagogical perspective, they make many key definitions and constructions less technical. They are also better-suited than operads for equivariant homotopy theory and its relatives.

Our main result establishes a universal property for the $\infty$-category of Lawvere theories, which completely characterizes the relationship between a Lawvere theory and its $\infty$-category of models. Many familiar properties of Lawvere theories follow directly.

As a consequence, we prove that the Burnside category is a classifying object for additive categories, as promised in an earlier paper, and as part of a more general correspondence between enriched Lawvere theories and module Lawvere theories.
\end{abstract}

\section{Introduction}
\noindent The primary goal of this paper is to lay the foundations for a Lawvere theoretic approach to higher algebra. As such, we organize the paper as a survey, combining others' work with our own. In this introduction, we summarize the new results for those who are familiar with the subject; however, the casual reader may be better served by reading Section 2, followed by this introduction.

A Lawvere theory $\mathcal{L}$ encodes a particular type of algebraic theory (for example, commutative monoids). A \emph{model} of $\mathcal{L}$ is an instance of that algebraic structure. As we are working with $\infty$-categories, we take models by default relative to the $\infty$-category Top of spaces.

Specifically, $\mathcal{L}$ is an $\infty$-category with finite products, generated by a single object, and models are functors $\mathcal{L}\rightarrow\text{Top}$ which preserve finite products.

The assignment from $\mathcal{L}$ to the $\infty$-category $\text{Mdl}_\mathcal{L}$ of models is functorial $$\text{Mdl}:\text{Lwv}\rightarrow\text{Pr}^L,$$ where an object of $\text{Pr}^L$ is a presentable $\infty$-category, and a morphism is a left adjoint functor. In fact, since any Lawvere theory $\mathcal{L}$ has a distinguished object $1$, $\text{Mdl}_\mathcal{L}$ also has a distinguished object, the model $\text{Map}_\mathcal{L}(1,-)$. Therefore, $\text{Mdl}$ may be promoted to a functor $\text{Lwv}\rightarrow\text{Pr}^L_\ast$.

\begin{theorem*}[Theorems \ref{Thm1}, \ref{Thm2}]
The functor $\text{Mdl}:\text{Lwv}\rightarrow\text{Pr}^L_\ast$ is fully faithful and symmetric monoidal, and it has a right adjoint which sends $\mathcal{C}\in\text{Pr}^L_\ast$ to $\mathcal{C}_\text{fgf}^\text{op}$, where $\mathcal{C}_\text{fgf}\subseteq\mathcal{C}$ is the full subcategory of \emph{finitely generated free} objects.
\end{theorem*}

\noindent In this way, Lwv is a \emph{symmetric monoidal colocalization} of the better understood $\infty$-category $\text{Pr}^L_\ast$, which allows us to study Lawvere theories using familiar tools like the adjoint functor theorem and Lurie's tensor products of categories (\cite{HA} 4.8).

The math here is not new, but the packaging is. That is, the theorem encapsulates the following facts about Lawvere theories:
\begin{itemize}
\item (the adjunction) $\text{Mdl}_\mathcal{L}$ is the \emph{free cocompletion} of $\mathcal{L}^\text{op}$, regarding the latter as an $\infty$-category with finite coproducts (\cite{HTT} 5.3.6.10);
\item (Mdl is fully faithful) $\mathcal{L}$ can be recovered from $\text{Mdl}_\mathcal{L}$ as the full subcategory of finitely generated free objects (Proposition \ref{PropYoneda});
\end{itemize}

\noindent The fact that Mdl is symmetric monoidal may be less familiar, but it is a consequence of Lurie's `tensor products of categories' machinery in \cite{HA} 4.8. In Section 3, we explore two direct corollaries of this monoidality:
\begin{itemize}
\item (Corollary \ref{Cor1}) If a Lawvere theory $\mathcal{L}$ admits a symmetric monoidal structure compatible with finite products, then $\text{Mdl}_\mathcal{L}$ inherits a closed symmetric monoidal structure called \emph{Day convolution};
\item (Corollary \ref{CorEnr}) If $\mathcal{L}$ is as above, then $\mathcal{L}$ is a semiring $\infty$-category. Any module over $\mathcal{L}$ is naturally enriched in $\text{Mdl}_\mathcal{L}$ with its Day convolution.
\end{itemize}

\noindent The first of these is well-known, but the second we believe is new. It also admits a partial converse:

\begin{theorem*}[Theorem \ref{ThmEnr}]
If $\mathcal{L}$ admits a symmetric monoidal structure compatible with finite products, and another Lawvere theory $\mathcal{K}$ is enriched in $\text{Mdl}_\mathcal{L}$, then $\mathcal{K}$ is naturally tensored over $\mathcal{L}$ via a map of Lawvere theories $$\mathcal{L}\otimes\mathcal{K}\rightarrow\mathcal{K}.$$
\end{theorem*}

\noindent This theorem suggests that there may be a strong converse to Corollary \ref{CorEnr}:

\begin{conjecture*}[Conjecture \ref{ConjLwv}]
If $\mathcal{L}$ is as above, the $\infty$-categories of $\text{Mdl}_\mathcal{L}$-enriched Lawvere theories and $\mathcal{L}$-module Lawvere theories are equivalent.
\end{conjecture*}

\noindent Sections 4 and 5 are devoted to examples and applications:

For any $\mathbb{E}_1$-semiring space $R$, there is a Lawvere theory whose models are $R$-modules. In Section 4, we show that these are the only semiadditive Lawvere theories; that is, semiadditive Lawvere theories can be identified with semiring spaces. This is an easy result for classical Lawvere theories. We record it here because it is slightly more subtle for $\infty$-categories. It also suggests a philosophy that we are fond of: Lawvere theories may be regarded as generalized (non-additive) rings.

In Section 5.1, we use our ideas relating module Lawvere theories to enriched Lawvere theories, proving a result promised in the author's earlier paper \cite{Berman}:

\begin{theorem*}[Theorem \ref{ThmBurn}]
The Burnside $\infty$-category is a commutative semiring $\infty$-category, and there is an equivalence $$\text{Mod}_\text{Burn}\cong\text{AddCat}_\infty,$$ where $\text{AddCat}_\infty$ denotes the $\infty$-category of additive $\infty$-categories.
\end{theorem*}

\noindent Finally, in Section 5.2, we describe an application to equivariant homotopy theory. There are no results; at this point, the application is just motivational.

\subsection{Acknowledgment}
\noindent This paper is largely drawn from the author's thesis \cite{Thesis}, of which it is the second part. It has been in the works for years, and benefitted from conversations with many others, including Ben Antieau, Clark Barwick, Saul Glasman, Rune Haugseng, Mike Hill, Bogdan Krstic, and others.

\section{Fundamentals of Lawvere theories}
\subsection{Lawvere theories and their models}
\noindent Classical Lawvere theories are one of the earliest formulations of algebraic theory, dating to Lawvere's 1963 thesis \cite{Lawvere}, and they have been thoroughly studied since then; see \cite{ARV}. In the setting of $\infty$-categories, Lawvere theories have been studied by just a few authors, notably by Cranch \cite{Cranch} and Gepner-Groth-Nikolaus \cite{GGN}, and in the prequel to this paper \cite{Berman}.

The literature in this area is sparse in part because Lurie's book Higher Algebra \cite{HA} founds the subject on operads, instead. That approach is now well-developed, and it has many benefits, but it suffers from the drawback of being pedagogically unwieldy. Even for those who are already invested in operads, like many homotopy theorists, there is a high barrier of entry: Even the definitions of fundamental objects like $\infty$-operads, commutative algebras, and symmetric monoidal $\infty$-categories presuppose a deep understanding of the quasicategory model.

On the other hand, there are surely those who would like to use this machinery without requiring operads at all.

In contrast, the Lawvere theoretic approach can be described in two sentences, and independent of our model of $\infty$-category:

\begin{definition}
A \emph{Lawvere theory} is an $\infty$-category $\mathcal{L}$ which is closed under finite products and generated under finite products by a single distinguished object $1$. An algebra or \emph{model} of $\mathcal{L}$ in $\mathcal{C}$ is a functor $\mathcal{L}\rightarrow\mathcal{C}$ which preserves finite products.
\end{definition}

\noindent Typically, we want to take $\mathcal{C}$ to be an $\infty$-category which is presentable and cartesian closed, such as Set, Top, or $\text{Cat}_\infty$, but there is no such requirement. By default, we take $\mathcal{C}$ to be the $\infty$-category Top of CW complexes, or homotopy types, which is the initial object among $\infty$-categories which are presentable and closed symmetric monoidal. We write $$\text{Mdl}_\mathcal{L}=\text{Fun}^\times(\mathcal{L},\text{Top}).$$

\begin{example}
Let Fin denote the category of finite sets. Then $\text{Fin}^\text{op}$ is a Lawvere theory, with product given by disjoint union of sets, and evaluation at the singleton $$\text{Mdl}_{\text{Fin}^\text{op}}\rightarrow\text{Top}$$ is an equivalence of $\infty$-categories. We say $\text{Fin}^\text{op}$ is the \emph{trivial Lawvere theory}.
\end{example}

\begin{example}
Let $\text{Burn}^\text{eff}$ denote the effective Burnside 2-category, whose objects are finite sets. For finite sets $X,Y$, the groupoid of morphisms from $X$ to $Y$ is the groupoid of span diagrams $X\leftarrow T\rightarrow Y$, and composition is via pullback. Then $\text{Burn}^\text{eff}$ is a Lawvere theory, with product given by disjoint union of sets.

If $f:\text{Burn}^\text{eff}\rightarrow\text{Top}$ is a model, then the spans $0\xleftarrow{=} 0\rightarrow 1$, respectively $2\xleftarrow{=} 2\rightarrow 1$ endow $f(1)$ with a distinguished point, respectively a binary operation $f(1)\times f(1)\rightarrow f(1)$. Composition in $\text{Burn}^\text{eff}$ precisely enforces the structure of a commutative (or $\mathbb{E}_\infty$) monoid on $f(1)$, and evaluation at $1$ $$\text{Mdl}_{\text{Burn}^\text{eff}}\rightarrow\text{CMon}_\infty$$ is an equivalence of $\infty$-categories. See \cite{Berman} Remark 3.6.

We say $\text{Burn}^\text{eff}$ is the commutative (or $\mathbb{E}_\infty$) Lawvere theory.
\end{example}

\noindent In fact, for any cartesian monoidal $\mathcal{C}^\times$, commutative monoids in $\mathcal{C}$ are equivalent to models of $\text{Burn}^\text{eff}$ in $\mathcal{C}$ (\cite{Berman} 3.6). Applying this to \cite{HA} 2.4.2.4:

\begin{example}\label{ExSMCat}
A symmetric monoidal $\infty$-category may be regarded as a functor $\text{Burn}^\text{eff}\rightarrow\text{Cat}_\infty$ which preserves finite products. In contrast to Lurie's definition of a symmetric monoidal $\infty$-category, this construction is elementary: In order to understand it, we need only understand how to take finite products in an $\infty$-category (they are the same as products in the homotopy category!), and how to regard a 2-category such as $\text{Burn}^\text{eff}$ as an $\infty$-category.
\end{example}

\noindent More generally, given any $\infty$-operad $\mathcal{O}$, there is a Lawvere theory $\mathcal{L}$ such that $$\text{Alg}_\mathcal{O}(\mathcal{C}^\times)\cong\text{Mdl}_\mathcal{L}(\mathcal{C}^\times)$$ for any cartesian monoidal $\mathcal{C}^\times$. The Lawvere theory can even be more-or-less explicitly described in terms of $\mathcal{O}$ (\cite{Berman} 3.16). \\

\noindent It may appear that Lawvere theories are less general than operads because they apply only to cartesian monoidal $\infty$-categories. This is apparently a significant obstacle: a major application of operads is to understanding multiplicative structure on rings. For example, a ring spectrum is an algebra in spectra under \emph{smash product}. However, this problem can be easily overcome, as long as we restrict attention to \emph{connective} ring spectra:

\begin{example}
For any connective ring spectrum $R$, there are Lawvere theories whose models are equivalent to $\text{Mod}_R^{\geq 0}$, $\text{Alg}_R^{\geq 0}$, and $\text{CAlg}_R^{\geq 0}$.
\end{example}

\noindent This example and others like it follow from a result of Gepner, Groth, and Nikolaus \cite{GGN} which describes \emph{exactly} which $\infty$-categories are equivalent to $\text{Mdl}_\mathcal{L}$ for some Lawvere theory $\mathcal{L}$ (Theorem \ref{ThmGGN} below).

Before stating their result, we recall the theory of presentable $\infty$-categories. By the adjoint functor theorem (\cite{HTT} 5.5.2.9), a functor $\mathcal{C}\rightarrow\mathcal{D}$ between presentable $\infty$-categories preserves small colimits if and only if it has a right adjoint. We write $\text{Pr}^L$ for the $\infty$-category of presentable $\infty$-categories along with these \emph{left adjoint} functors.

If $\mathcal{C}\in\text{Pr}^L$, then since Top is freely generated by one object under colimits, the following data are equivalent:
\begin{itemize}
\item an object $X\in\mathcal{C}$ (we say $\mathcal{C}$ is \emph{pointed});
\item a left adjoint functor $L=-\otimes X:\text{Top}\rightarrow\mathcal{C}$ (we say $L(S)$ is the \emph{free} object on $S$);
\item a right adjoint functor $R=\text{Map}(X,-):\mathcal{C}\rightarrow\text{Top}$ (we say $R(Y)$ is the \emph{underlying space} of $Y$).
\end{itemize}

\noindent We will denote by $\text{Pr}^L_\ast$ the $\infty$-category of these pointed presentable $\infty$-categories, along with left adjoint basepoint-preserving functors. (Formally, $\text{Pr}^L_\ast\cong\text{Pr}^L_{\text{Top}/}$.)

If $\mathcal{L}$ is a Lawvere theory, generated by the distinguished object $1$, then we regard $\text{Mdl}_\mathcal{L}$ as canonically pointed by the right adjoint forgetful functor $$\text{evaluate at 1}:\text{Mdl}_\mathcal{L}\rightarrow\text{Top}.$$ By the Yoneda lemma, the corresponding basepoint is the model $$\text{Map}(1,-):\mathcal{L}\rightarrow\text{Top}.$$

\begin{theorem}[Gepner-Groth-Nikolaus \cite{GGN} Theorem B.7]\label{ThmGGN}
A pointed presentable $\infty$-category $\mathcal{C}$ is equivalent to $\text{Mdl}_\mathcal{L}$ for some Lawvere theory $\mathcal{L}$ if and only if the forgetful functor $\mathcal{C}\rightarrow\text{Top}$ is conservative and preserves geometric realizations.
\end{theorem}

\subsection{Reconstructing Lawvere theories from their models}
\noindent We have just seen that many $\infty$-categories $\mathcal{M}$ can be described as models over a Lawvere theory (roughly, those which are presentable and \emph{algebraic} in nature). We may now ask: is that Lawvere theory unique, and to what extent can it be recovered from $\mathcal{M}$?

If $\mathcal{L}$ admits finite products, then $\text{Mdl}_\mathcal{L}=\text{Fun}^\times(\mathcal{L},\text{Top})$ is a full subcategory, by definition, of the $\infty$-category of presheaves, $\mathcal{P}(\mathcal{L}^\text{op})=\text{Fun}(\mathcal{L},\text{Top})$. Moreover, every representable presheaf $\text{Map}(X,-):\mathcal{L}\rightarrow\text{Top}$ preserves any limits that exist in $\mathcal{L}$, and therefore is in $\text{Mdl}_\mathcal{L}$.

By the Yoneda lemma, then $\mathcal{L}^\text{op}$ is a full subcategory of $\text{Mdl}_\mathcal{L}$. (We called this a \emph{cartesian monoidal Yoneda lemma} in \cite{Berman} 3.7.)

If $\mathcal{L}$ is a Lawvere theory, we can explicitly identify $\mathcal{L}^\text{op}$ as a subcategory of $\text{Mdl}_\mathcal{L}$: The embedding $\mathcal{L}^\text{op}\subseteq\text{Mdl}_\mathcal{L}$ identifies $1^{\amalg n}\in\mathcal{L}^\text{op}$ with $\mathbb{I}^{\amalg n}\in\text{Mdl}_\mathcal{L}$. Here $\mathbb{I}$ is the distinguished object of $\text{Mdl}_\mathcal{L}$, so that $\mathbb{I}^{\amalg n}$ can also be regarded as the free model on $n$ generators. In conclusion:

\begin{proposition}\label{PropYoneda}
Suppose $\mathcal{M}$ is a pointed, presentable $\infty$-category with distinguished object $\mathbb{I}$. Let $\mathcal{M}_\text{fgf}$ be the full subcategory of finitely generated free objects; that is, those of the form $\mathbb{I}^{\amalg n}$ for integers $n\geq 0$. If $\mathcal{M}\cong\text{Mdl}_\mathcal{L}$ for some Lawvere theory $\mathcal{L}$, then $\mathcal{L}\cong\mathcal{M}_\text{fgf}^\text{op}$.
\end{proposition}

\noindent Theoretically, this proposition combined with Theorem \ref{ThmGGN} allow us to describe the Lawvere theories associated to any operad, modules over a ring, algebras over a ring, etc. -- \emph{provided we already understand the $\infty$-category of models}.

However, we may instead seek to describe the Lawvere theory first, and use it to construct some new $\infty$-category of models. (For example, we might want to do this for pedagogical purposes as in Example \ref{ExSMCat}.)

\begin{principle}\label{PrCombLwv}
Lawvere theories often have combinatorial descriptions, in which their objects are finite sets, morphisms are given by diagrams of finite sets, and products are given by disjoint union.
\end{principle}

\begin{example}
The commutative Lawvere theory $\text{Burn}^\text{eff}$ is equivalent to the 2-category of spans of finite sets.
\end{example}

\noindent We may revisit this principle in a future paper on \emph{combinatorial Lawvere theories}; for now, we will not emphasize it.

\subsection{Main theorem}
\noindent We have described how to pass back and forth between a Lawvere theory and its $\infty$-category of models. We will show this relation is exceptionally robust.

Let Lwv denote the $\infty$-category whose objects are Lawvere theories and morphisms are functors which preserve finite products and the distinguished object\footnote{Formally, Lwv is a full subcategory of pointed cartesian monoidal $\infty$-categories.}.

\begin{theorem}\label{Thm1}
There is an adjunction $$\text{Mdl}:\text{Lwv}\leftrightarrows\text{Pr}^L_\ast:(-)^\text{op}_\text{fgf},$$ and the left adjoint Mdl is fully faithful.
\end{theorem}

\noindent In other words, Lwv is a \emph{colocalization} of $\text{Pr}^L_\ast$. Theorem \ref{ThmGGN} described explicitly \emph{which} colocalization by providing a testable criterion to determine the essential image of Lwv in $\text{Pr}^L_\ast$.

\begin{proof}
Call $L=\text{Mdl}$ and $R=(-)^\text{op}_\text{fgf}$, suggestive of left and right adjoints. The composition $RL$ is equivalent to the identity by Proposition \ref{PropYoneda}. Therefore, applying $R$ induces a map of spaces $$\text{Fun}^L_\ast(L(\mathcal{L}),\mathcal{C})\xrightarrow{R_\ast}\text{Fun}^\times_\ast(\mathcal{L},R(\mathcal{C})),$$ natural in both $\mathcal{L}\in\text{Lwv}$ and $\mathcal{C}\in\text{Pr}^L_\ast$, and this $R_\ast$ is a natural isomorphism by \cite{HTT} 5.3.6.10. Therefore, $L$ and $R$ are adjoint.

For $\mathcal{L},\mathcal{K}\in\text{Lwv}$, applying $L$ induces a map of spaces $$\text{Fun}^\times_\ast(\mathcal{L},\mathcal{K})\xrightarrow{L_\ast}\text{Fun}^L_\ast(L(\mathcal{L}),L(\mathcal{K})).$$ Since $RL\cong\text{Id}$, $R_\ast L_\ast$ is equivalent to the identity, and because $R_\ast$ is an equivalence (as above), so is $L_\ast$. Therefore, $L$ is fully faithful.
\end{proof}

\begin{remark}
As seen in the proof, Theorem \ref{Thm1} combines two facts in one: the adjunction asserts \cite{HTT} 5.3.6.10, that $\text{Mdl}_\mathcal{L}$ is the \emph{free cocompletion} of $\mathcal{L}^\text{op}$ (regarding the latter as an $\infty$-category which already has finite coproducts).

That the left adjoint is fully faithful asserts Proposition \ref{PropYoneda}, that $\mathcal{L}^\text{op}$ is a full subcategory of $\text{Mdl}_\mathcal{L}$.
\end{remark}

\begin{remark}
More generally, if $\text{CartMonCat}_\infty$ denotes the $\infty$-category of $\infty$-categories which admit finite products (and functors which preserve them), then the same proof shows $$\text{Mdl}:\text{CartMonCat}_\infty\rightarrow\text{Pr}^L$$ is fully faithful. It also comes very close to being a left adjoint to the functor $$(-)^\text{op}:\text{Pr}^L\rightarrow\widehat{\text{CartMonCat}}_\infty,$$ with the following fatal obstacle: The domain of Mdl consists of \emph{small} cartesian monoidal $\infty$-categories, while the codomain of $(-)^\text{op}$ consists of \emph{large} cartesian monoidal $\infty$-categories.

This set-theoretic problem arises because of the definition of presentable $\infty$-categories: they are required to have a small set of generating objects, but these objects are not remembered as part of the data. Theorem \ref{Thm1} solves the problem by introducing a \emph{framing}; that is, by remembering these objects (or in this case, a single object).
\end{remark}

\section{Algebra of Lawvere theories}
\noindent In Section \ref{SecSM1}, we show that the functor $\text{Mdl}_\mathcal{L}:\text{Lwv}\rightarrow\text{Pr}^L_\ast$ is symmetric monoidal. Then we study its behavior on commutative algebras in \ref{SecSM2}, constructing Day convolution products of models. Finally, we study its behavior on modules in Section \ref{SecSM3}, showing that many Lawvere theories have canonical enrichments. In fact, we provide evidence that $\mathcal{L}$-module Lawvere theories can be identified with $\text{Mdl}_\mathcal{L}$-enriched Lawvere theories (Conjecture \ref{ConjLwv}).

\subsection{Kronecker products of Lawvere theories}\label{SecSM1}
\noindent The primary technical contribution of this paper is first to cast the relationship between Lawvere theories and their models as a colocalization (Theorem \ref{Thm1}) and second that this colocalization is compatible with symmetric monoidal structures on Lwv (the Kronecker product) and $\text{Pr}^L_\ast$ (Lurie's tensor product). We review each of these:

\begin{remark}[Lurie's tensor product of presentable $\infty$-categories]
There is a closed symmetric monoidal tensor product on $\text{Pr}^L$ with the following universal property: left adjoint functors $\mathcal{C}\otimes\mathcal{D}\rightarrow\mathcal{E}$ can be identified with functors $\mathcal{C}\times\mathcal{D}\rightarrow\mathcal{E}$ which preserve small colimits in each variable separately. This is constructed by Lurie in \cite{HA} 4.8.1, and the unit is Top.

If $\mathcal{V}$ is presentable, to endow $\mathcal{V}$ with the structure of a commutative algebra in $\text{Pr}^{L,\otimes}$ is precisely to endow $\mathcal{V}$ with its own closed symmetric monoidal structure\footnote{a consequence of the adjoint functor theorem.}.
\end{remark}

\noindent If $\mathcal{C},\mathcal{D}$ are two \emph{pointed} presentable $\infty$-categories, $\mathcal{C}\otimes\mathcal{D}$ is also canonically pointed (for example, by the free functor $\text{Top}\cong\text{Top}\otimes\text{Top}\rightarrow\mathcal{C}\otimes\mathcal{D}$), so that $\text{Pr}^L_\ast$ is also symmetric monoidal via Lurie's tensor product.

A commutative algebra object in $\text{Pr}^L_\ast$ is a presentable $\infty$-category with a closed symmetric monoidal structure $\otimes$, pointed by the unit of $\otimes$.

\begin{remark}[Kronecker tensor product of Lawvere theories]
There is a closed symmetric monoidal tensor product of cartesian monoidal $\infty$-categories with the following universal property: functors $\mathcal{C}\otimes\mathcal{D}\rightarrow\mathcal{E}$ which preserve finite products can be identified with functors $\mathcal{C}\times\mathcal{D}\rightarrow\mathcal{E}$ which preserve finite products in each variable separately. This can be made precise in two equivalent ways:
\begin{itemize}
\item by Lurie's general framework of tensor products of categories (\cite{HA} 4.8.1);
\item as in \cite{Berman}: cartesian monoidal $\infty$-categories can be identified with modules over the commutative semiring category $\text{Fin}^\text{op}$, which admit a relative tensor product $\otimes_{\text{Fin}^\text{op}}$.
\end{itemize}
\noindent If $\mathcal{C},\mathcal{D}$ are Lawvere theories (that is, generated by a single object under $\times$), then $\mathcal{C}\otimes\mathcal{D}$ is also a Lawvere theory, so that Lwv inherits a symmetric monoidal operation $\otimes$ called \emph{Kronecker product}, and the unit is $\text{Fin}^\text{op}$. See \cite{Berman} for details.
\end{remark}

\noindent For classical Lawvere theories, the Kronecker product goes back to Freyd \cite{KroneckerProduct}, and it is also compatible with the Boardman-Vogt tensor product (\cite{HA} 2.2.5) of operads.

That is, if $\mathcal{L}_\mathcal{O}$ is the Lawvere theory associated to an operad $\mathcal{O}$, then $\mathcal{L}_{\mathcal{O}\otimes\mathcal{O}^\prime}\cong\mathcal{L}_\mathcal{O}\otimes\mathcal{L}_{\mathcal{O}^\prime}$, because the two sides have equivalent $\infty$-categories of models.

\begin{theorem}\label{Thm2}
The functors $$\text{Mdl}:\text{CartMonCat}_\infty\rightarrow\text{Pr}^L$$ $$\text{Mdl}:\text{Lwv}\rightarrow\text{Pr}^L_\ast$$ are compatible with the symmetric monoidal structures of the last two remarks.
\end{theorem}

\noindent That is, Lwv is a \emph{symmetric monoidal colocalization} of $\text{Pr}^L_\ast$.

In the next two sections, we will explore the consequences of this theorem when applied to (first) commutative algebras and (second) modules with respect to the two symmetric monoidal structures.

\begin{proof}
By \cite{HA} 4.8.1.8, the functor $$\text{Fun}^\times((-)^\text{op},\text{Top}):\text{CocartMonCat}_\infty\rightarrow\text{Pr}^L$$ is symmetric monoidal. Passing via the symmetric monoidal equivalence $(-)^\text{op}:\text{CartMonCat}_\infty\rightarrow\text{CocartMonCat}_\infty$, we have the first part.

Taking $\mathbb{E}_0$-algebras (that is, pointed objects) on each side, we see that $\text{Mdl}:\text{CartMonCat}_\ast\rightarrow\text{Pr}^L_\ast$ is symmetric monoidal, and this restricts to the symmetric monoidal full subcategory $\text{Lwv}\subseteq\text{CartMonCat}_\ast$.
\end{proof}

\subsection{Algebra Lawvere theories and Day convolution}\label{SecSM2}
\noindent For suitable Lawvere theories $\mathcal{L}$, we can use Theorem \ref{Thm2} to construct tensor products of $\mathcal{L}$-models. A commutative algebra structure on $\mathcal{L}\in\text{Lwv}^\otimes$ amounts to a symmetric monoidal structure on $\mathcal{L}$ which preserves finite products independently in each variable, and such that the unit is the distinguished object $1$.

Such Lawvere theories are sometimes called \emph{commutative algebraic theories} in the classical literature \cite{Keigher}.

\begin{corollary}\label{Cor1}
If $\mathcal{L}\in\text{CAlg}(\text{Lwv}^\otimes)$, then $\text{Mdl}_\mathcal{L}$ inherits a closed symmetric monoidal structure called \emph{Day convolution}, with unit $\mathbb{I}$\footnote{This is the Day convolution of Lurie \cite{HA} and Glasman \cite{DayCon}.}.

Conversely, if $\text{Mdl}_\mathcal{L}$ has a closed symmetric monoidal structure with unit $\mathbb{I}$, then $\mathcal{L}$ inherits a commutative algebra structure in Lwv.
\end{corollary}

\begin{proof}
Since $\text{Mdl}$ is symmetric monoidal, it takes commutative algebras to commutative algebras. Therefore, if $\mathcal{L}\in\text{CAlg}(\text{Lwv}^\otimes)$, then $\text{Mdl}_\mathcal{L}$ has a commutative algebra structure in $\text{Pr}^L_\ast$, which is to say a closed symmetric monoidal structure with unit $\mathbb{I}$.

Conversely, the right adjoint to a symmetric monoidal functor is lax symmetric monoidal \cite{GepnerHaugseng}. If $\text{Mdl}_\mathcal{L}$ has a closed symmetric monoidal structure with unit $\mathbb{I}$, it is a commutative algebra in $\text{Pr}^L_\ast$, so $\mathcal{L}\in\text{CAlg}(\text{Lwv}^\otimes)$.
\end{proof}

\begin{example}
The effective Burnside 2-category is symmetric monoidal under cartesian product, which makes it a commutative algebraic theory. Therefore, Day convolution provides a closed symmetric monoidal \emph{smash product} for $\mathbb{E}_\infty$-spaces.
\end{example}

\begin{remark}[Models in other $\infty$-categories]
More generally, suppose $\mathcal{L}$ is a Lawvere theory, and $\mathcal{V}$ is a presentable $\infty$-category. By general theory of presentable $\infty$-categories, $$\text{Fun}^\times(\mathcal{L},\mathcal{V})\cong\text{Fun}^\amalg(\mathcal{L}^\text{op},\mathcal{V}^\text{op})\cong\text{Fun}^L(\text{Mdl}_\mathcal{L},\mathcal{V}^\text{op})\cong\text{Fun}^R(\mathcal{V}^\text{op},\text{Mdl}_\mathcal{L}).$$ Lurie proves (\cite{HA} 4.8.1.17) for two presentable $\infty$-categories $\mathcal{C}$ and $\mathcal{D}$, that $\mathcal{C}\otimes\mathcal{D}\cong\text{Fun}^R(\mathcal{C}^\text{op},\mathcal{D})$. Therefore, models of $\mathcal{L}$ in $\mathcal{V}$ can be identified with the tensor product $$\text{Mdl}_\mathcal{L}(\mathcal{V})\cong\text{Mdl}_\mathcal{L}\otimes\mathcal{V}.$$ This equivalence is due to \cite{GGN} Proposition B.3.

In particular, if $\mathcal{L}\in\text{CAlg}(\text{Lwv}^\otimes)$ and $\mathcal{V}$ has a closed symmetric monoidal structure, then $\text{Mdl}_\mathcal{L}(\mathcal{V})$ also has a closed symmetric monoidal structure\footnote{because tensor products of commutative algebras are commutative algebras} (Day convolution).
\end{remark}

\subsection{Module Lawvere theories and enrichment}\label{SecSM3}
\noindent Let $\mathcal{V}$ be presentable and closed symmetric monoidal; i.e., $\mathcal{V}\in\text{CAlg}(\text{Pr}^{L,\otimes})$. If $\mathcal{M}\in\text{Pr}^L$ is a $\mathcal{V}$-module, and $X\in\mathcal{M}$, then $-\otimes X:\mathcal{V}\rightarrow\mathcal{M}$ has a right adjoint $\text{Map}(X,-):\mathcal{M}\rightarrow\mathcal{V}$ which makes $\mathcal{M}$ naturally $\mathcal{V}$-enriched. Gepner and Haugseng have made this precise (\cite{GepnerHaugseng} 7.4.13).

\begin{conjecture}
Conversely, we may think of $\mathcal{V}$-modules in $\text{Pr}^L$ as precisely those $\mathcal{V}$-enriched categories which are \emph{presentable in an enriched sense}. As far as the author is aware, the notion of `presentable in an enriched sense' has not yet been made rigorous for $\infty$-categories, but this is a philosophy already familiar to some.
\end{conjecture}

\noindent We have a second corollary of Theorem \ref{Thm2}:

\begin{corollary}\label{CorEnr}
If $\mathcal{L}$ is a commutative semiring $\infty$-category whose additive structure is cartesian monoidal, and $\mathcal{M}$ is an $\mathcal{L}$-module, then $\mathcal{M}$ is naturally enriched in $\text{Mdl}_\mathcal{L}$.
\end{corollary}

\begin{proof}
If $\mathcal{M}$ is an $\mathcal{L}$-module in Lwv, then $\text{Mdl}_\mathcal{M}$ is a $\text{Mdl}_\mathcal{L}$-model in $\text{Pr}^L$. As above, $\text{Mdl}_\mathcal{M}$ inherits a canonical $\text{Mdl}_\mathcal{L}$-enrichment, which restricts to a $\text{Mdl}_\mathcal{L}$-enrichment on the full subcategory $\mathcal{M}\subseteq\text{Mdl}_\mathcal{M}^\text{op}$.
\end{proof}

\begin{example}\label{ExSemi}
If $\mathcal{L}=\text{Burn}^\text{eff}$ is the Lawvere theory for $\mathbb{E}_\infty$-spaces, then $\text{Burn}^\text{eff}$-modules can be identified with semiadditive $\infty$-categories (\cite{Berman} Theorem 1.2). By corollary \ref{CorEnr}, any semiadditive $\infty$-category is naturally enriched in $\mathbb{E}_\infty$-spaces.

This is the homotopical analogue of a classical fact: semiadditive categories are naturally enriched in commutative monoids.
\end{example}

\noindent Conversely, if $\mathcal{L}\in\text{CAlg}(\text{Lwv}^\otimes)$, then any Lawvere theory enriched in $\text{Mdl}_\mathcal{L}$ is naturally tensored over $\mathcal{L}$:

\begin{definition}
If $\mathcal{V}$ is presentable and closed symmetric monoidal, a \emph{$\mathcal{V}$-enriched Lawvere theory} is a $\mathcal{V}$-enriched category for which the underlying $\infty$-category is a Lawvere theory.
\end{definition}

\begin{theorem}\label{ThmEnr}
Suppose $\mathcal{L}\in\text{CAlg}(\text{Lwv}^\otimes)$, and $\mathcal{K}^\text{enr}$ is a $\text{Mdl}_\mathcal{L}$-enriched Lawvere theory with underlying $\infty$-category $\mathcal{K}$. For any $X\in\mathcal{L}$, $K\in\mathcal{K}$, there is some object $X\otimes K\in\mathcal{K}$ with a natural isomorphism $$\text{Map}^\text{enr}(-,K)(X)\cong\text{Map}(-,X\otimes K),$$ and this $\otimes$ arises from a morphism of Lawvere theories $$\mathcal{L}\otimes\mathcal{K}\rightarrow\mathcal{K}.$$
\end{theorem}

\begin{proof}
Because $\mathcal{K}^\text{enr}$ is $\text{Mdl}_\mathcal{L}$-enriched, there is a functor $$\text{Map}^\text{enr}_\mathcal{K}(-,-):\mathcal{K}^\text{op}\times\mathcal{K}\rightarrow\text{Mdl}_\mathcal{L},$$ and the composite with the forgetful functor $\text{ev}_1:\text{Mdl}_\mathcal{L}\rightarrow\text{Top}$ is the ordinary mapping space $\text{Map}_\mathcal{K}$. In particular, that composite preserves finite products in the second variable.

By Theorem \ref{ThmGGN}, the forgetful functor $\text{ev}_1$ (evaluation at 1) is conservative. It also preserves finite products because it is a right adjoint. Therefore, the functor $\text{Map}^\text{enr}(-,-)$ preserves finite products in the second variable.

It follows that the adjoint $\mathcal{L}\times\mathcal{K}\rightarrow\text{Fun}(\mathcal{K}^\text{op},\text{Top})$ preserves finite products independently in each variable, thereby inducing a functor $$\phi:\mathcal{L}\otimes\mathcal{K}\rightarrow\text{Fun}(\mathcal{K}^\text{op},\text{Top})$$ which preserves finite products. By construction, $\phi(X\otimes K)\cong\text{Map}^\text{enr}(-,K)(X)$. In particular, if $n$ denotes $1^{\times n}$ in a Lawvere theory, $$\phi(n)\cong\text{Map}^\text{enr}_\mathcal{K}(-,n)(1)\cong\text{ev}_1\circ\text{Map}^\text{enr}_\mathcal{K}(-,n)\cong\text{Map}_\mathcal{K}(-,n),$$ so that $\phi$ factors through the full subcategory $\mathcal{K}\subseteq\text{Fun}(\mathcal{K}^\text{op},\text{Top})$. This completes the proof.
\end{proof}

\noindent Together, Corollary \ref{CorEnr} and Theorem \ref{ThmEnr} are suggestive of:

\begin{conjecture}\label{ConjLwv}
If $\mathcal{L}\in\text{CAlg}(\text{Lwv}^\otimes)$, the $\infty$-categories of $\text{Mdl}_\mathcal{L}$-enriched Lawvere theories and $\mathcal{L}$-module Lawvere theories are equivalent.
\end{conjecture}

\section{Additive Lawvere theories}
\noindent Suppose that $\mathcal{L}$ is semiadditive as an $\infty$-category; essentially, finite products are also finite coproducts. By Example \ref{ExSemi}, $\mathcal{L}$ is enriched in $\mathbb{E}_\infty$-spaces.

Therefore, $\text{End}(1)$ is an $\mathbb{E}_1$-semiring space, where $1$ is the distinguished object of $\mathcal{L}$, and there is a functor $$\text{End}(1):\text{SemiaddLwv}\rightarrow\mathbb{E}_1\text{Semiring}.$$

\begin{proposition}\label{PropSALwv}
This functor is an equivalence of symmetric monoidal $\infty$-categories, identifying a semiring $R$ with the Lawvere theory modeling $\text{Mod}_R$.
\end{proposition}

\noindent The proposition should not be surprising; if $\mathcal{L}$ is semiadditive, then we know $$\text{Map}(1^{\times m},1^{\times n})\cong\text{Map}(1,1)^{\times mn},$$ which depends only on the semiring $\text{End}(1)$. Hence, we expect $\text{End}(1)$ to encode all the data of the Lawvere theory.

\begin{proof}
If $R$ is an $\mathbb{E}_1$-semiring space, write $\text{Burn}_R$ for the Lawvere theory whose models are $\text{Burn}_R$, which exists by Theorem \ref{ThmGGN}. We claim that any semiadditive Lawvere theory $\mathcal{L}$ is equivalent to $\text{Burn}_{\text{End}(1)}$.

Fix $\mathcal{L}$ and write $R=\text{End}(1)$. Since $R$ acts on $\text{Map}_{\text{Mdl}_\mathcal{L}}(1,-)$, we have $$\text{Map}(1,-):\text{Mdl}_\mathcal{L}\rightarrow\text{Mod}_R$$ which maps $1$ to $1$, preserves finite direct sums, and therefore restricts to a map of Lawvere theories $\alpha:\mathcal{L}\rightarrow\text{Burn}_R$.

Certainly $\alpha$ is essentially surjective, so we show it is full and faithful. Given objects $m=1^{\amalg m}$ and $n=1^{\amalg n}$ in $\mathcal{L}$, we wish to prove that $$\text{Map}_\mathcal{L}(m,n)\xrightarrow{\alpha_\ast}\text{Map}_{\text{Burn}_R}(m,n)$$ is an equivalence. When $m=n=1$, this is true by construction. Otherwise, since $\mathcal{L}$ and $\text{Burn}_R$ are semiadditive, we know on both sides that $\text{Map}(m,n)\cong R^{mn}$, so $\alpha_\ast$ is an equivalence.

Therefore, every semiadditive Lawvere theory is of the form $\text{Burn}_R$. By Theorem \ref{Thm1}, $\text{SemiaddLwv}$ is the symmetric monoidal full subcategory of $\text{Pr}^L_\ast$ spanned by the $\infty$-categories $\text{Mod}_R$, as $R$ ranges over $\mathbb{E}_1$-semiring spaces. On the other hand, the functor $$\mathbb{E}_1\text{Semiring}\rightarrow\text{Pr}^L_\ast$$ which sends $R$ to $\text{Mod}_R$ is also fully faithful and symmetric monoidal, which completes the proof.
\end{proof}

\noindent From the proposition, we deduce an important philosophy: Lawvere theories are more like \emph{algebraic} objects than \emph{categorical} objects. We might even regard them as generalized (non-additive) rings. 

Finally, we apply Theorem \ref{ThmGGN} to deduce:

\begin{corollary}
If $\mathcal{M}$ is presentable and semiadditive, and $\mathcal{M}\rightarrow\text{Top}$ is a right adjoint functor which is conservative and preserves geometric realizations, then $\mathcal{M}\cong\text{Mod}_R$ for some $\mathbb{E}_1$-semiring space $R$, compatibly with the forgetful functor $\text{Mod}_R\rightarrow\text{Top}$.
\end{corollary}

\section{Applications of Lawvere theories}
\noindent We will end with two applications. The first is to the commutative algebra of semiring $\infty$-categories, and the second to equivariant homotopy theory.

\subsection{Commutative algebra of categories}
\begin{definition}
An $\infty$-category is \emph{additive} if it is semiadditive and each mapping $\mathbb{E}_\infty$-space is grouplike.
\end{definition}

\noindent That is, an additive $\infty$-category is semiadditive and enriched in $\text{Ab}_\infty\cong\text{Sp}_{\geq 0}$ (grouplike $\mathbb{E}_\infty$-spaces, or connective spectra). As promised in \cite{Berman}, we prove:

\begin{theorem}\label{ThmBurn}
The Burnside $\infty$-category Burn (which is the Lawvere theory for connective spectra) is a commutative semiring $\infty$-category, and $\text{Mod}_\text{Burn}$, $\text{AddCat}_\infty$ are equivalent full subcategories of $\text{SymMonCat}_\infty$.
\end{theorem}

\begin{remark}
Note that Theorem \ref{ThmBurn} implies Conjecture \ref{ConjLwv} for the specific Lawvere theory Burn. Indeed, $\text{Mdl}_\text{Burn}\cong\text{Sp}_{\geq 0}$-enriched Lawvere theories are additive Lawvere theories, so the conjecture asserts that a Lawvere theory is a Burn-module if and only if it is additive.
\end{remark}

\begin{proof}
As in \cite{GGN}, the product map $\text{Sp}_{\geq 0}\otimes\text{Sp}_{\geq 0}\rightarrow\text{Sp}_{\geq 0}$ is an equivalence. Identifying $\text{Sp}_{\geq 0}$ with $\text{Mdl}_\text{Burn}$ and noting that $\text{Mdl}$ is symmetric monoidal, we find that $\text{Burn}\otimes\text{Burn}\rightarrow\text{Burn}$ is also an equivalence.

This means that the forgetful functor $\text{Mod}_\text{Burn}\rightarrow\text{SymMonCat}_\infty$ is fully faithful. We need to show that a symmetric monoidal $\infty$-category admits the structure of a Burn-module if and only if it is additive.

First, if $\mathcal{C}$ is a Burn-module, it is a $\text{Burn}^\text{eff}$-module and therefore semiadditive (Example \ref{ExSemi}), but it is also $\text{Mdl}_\text{Burn}$-enriched by Corollary \ref{CorEnr}. By definition, it is therefore additive.

Conversely, suppose $\mathcal{C}$ is additive, so that $\mathcal{P}(\mathcal{C})=\text{Fun}(\mathcal{C}^\text{op},\text{Top})$ is additive and presentable. By \cite{GGN}, $\mathcal{P}(\mathcal{C})$ is a $\text{Mdl}_\text{Burn}$-module in $\text{Pr}^L$, and therefore also in $\text{SymMonCat}_\infty$, because the functor $\text{Pr}^L\rightarrow\text{SymMonCat}_\infty$ is lax symmetric monoidal which forgets everything except the cocartesian monoidal structure.

The embedding $\text{Burn}\cong\text{Burn}^\text{op}\subseteq\text{Mdl}_\text{Burn}$ respects both symmetric monoidal structures ($\oplus$ and $\otimes$), so $\mathcal{P}(\mathcal{C})$ is a Burn-module, as a cocartesian monoidal $\infty$-category. Moreover, the full subcategory $\mathcal{C}\in\mathcal{P}(\mathcal{C})$ is closed under direct sum, so it inherits a Burn-module structure. This completes the proof.
\end{proof}

\subsection{Equivariant homotopy theory}
\noindent Throughout this section, $G$ is a finite group. We write $\text{Fin}_G$ for the category of finite $G$-sets, and $\text{Burn}_G$ for the associated $\infty$-category of virtual spans, often referred to as the \emph{Burnside $\infty$-category} without mention of the particular group. See \cite{BarMack} for more on this.

All group actions will be on the right.

There are two classical model categories of equivariant $G$-spaces: the `naive' model structure has weak equivalences those maps which are weak equivalences of the underlying space. The corresponding $\infty$-category is $\text{Fun}(BG,\text{Top})$, because equivalences in a functor $\infty$-category are likewise checked objectwise, and $BG$ has only one object (up to equivalence).

On the other hand, the `genuine' model structure has weak equivalences those maps which have inverses up to homotopy. This model category corresponds to an $\infty$-category $\text{Top}_G$ which is certainly not equivalent to $\text{Fun}(BG,\text{Top})$! For example, the map $EG\rightarrow\ast$ is an equivalence in the former but not in the latter model structure.

For spectra as well, there is a distinction between $\text{Fun}(BG,\text{Sp})$ and the $\infty$-category $\text{Sp}_G$ of genuine equivariant spectra. Consult \cite{Equivariant} for a classical survey.

We might ask how to describe $\text{Top}_G$ and $\text{Sp}_G$ in higher categorical terms. For this, we have the two theorems:
\begin{itemize}
\item (Elmendorf's Theorem: \cite{Elmendorf} Theorem 1) $\text{Top}_G\cong\text{Mdl}(\text{Fin}_G^\text{op})$;
\item (Guillou-May's Theorem: \cite{GMay2} Theorem 0.1, \cite{BarMack} Example B.6) \\$\text{Sp}_G^{\geq 0}\cong\text{Mdl}(\text{Burn}_G)$.
\end{itemize}

\noindent Recall that we have used the notation $\text{Mdl}(\mathcal{L})=\text{Fun}^\times(\mathcal{L},\text{Top})$ whenever $\mathcal{L}$ admits finite products, even if it is not a Lawvere theory. However, $\text{Fin}_G$ and $\text{Burn}_G^\text{eff}$ are not far from being Lawvere theories: although they do not have single generating objects, they are generated freely by the set of orbits $G/H$, as $H$ ranges over subgroups of $G$.

We call them \emph{colored Lawvere theories}, with set of colors $\{G/H\}$, or \emph{equivariant Lawvere theories}, because they admit essentially surjective, product-preserving maps from the groupoid of finite $G$-sets, $\text{Fin}_G^\text{iso}$.

\begin{remark}\label{RmkGen}
The word `genuine', used to describe equivariant spaces and spectra, can be misleading. Frequently, group actions on spectra arise via abstract homotopy-theoretic means, such as when the spectra themselves are algebraic in nature (as in chromatic homotopy theory). In these cases, we typically do not expect `genuine' equivariant structures.

However, when our spaces or spectra arise geometrically out of point-set constructions, group actions will be `genuine'. This is because we can pass through the model category of genuine equivariant objects, on our way to the abstract $\infty$-categories $\text{Top}_G$ and $\text{Sp}_G$.

It would almost be better to regard the `genuine' actions as `geometric', and the `naive' actions as `homotopical'.
\end{remark}

\noindent The theorems of Elmendorf and Guillou-May may be combined with Corollary \ref{CorEnr} as follows:

\begin{corollary}
Regarding $\text{Fin}_G$ and $\text{Burn}_G^\text{eff}$ as commutative semiring $\infty$-categories, any $\text{Fin}_G$-module is naturally enriched in genuine $G$-spaces, and any $\text{Burn}_G^\text{eff}$-module is naturally enriched in (connective) genuine $G$-spectra.
\end{corollary}

\noindent More generally, suppose we have some algebraic structure, whose homotopical instances form an $\infty$-category $\mathcal{C}$. For example, $\mathcal{C}=\text{Sp}_{\geq 0}$ corresponds to the structure `abelian group'. We might ask: what kind of structure does a \emph{genuine} equivariant $G$-object of $\mathcal{C}$ have?

This is a question which is not entirely idle. Following Remark \ref{RmkGen}, if an object of $\mathcal{C}$ has an action of $G$ at some sufficiently concrete point-set level, we might expect \emph{additional structure} to carry over to the $\infty$-category $\mathcal{C}$, beyond a naive $G$-action.

By analogy with the theorems of Elmendorf and Guillou-May, we propose addressing this question via a 3-step procedure:
\begin{enumerate}
\item check whether $\mathcal{C}$ is of the form $\text{Mdl}_\mathcal{L}$ for some Lawvere theory $\mathcal{L}$ (possibly by means of Theorem \ref{ThmGGN});
\item check whether $\mathcal{L}$ can be described combinatorially, by applying some construction $\mathcal{M}$ to Fin (as in Principle \ref{PrCombLwv});
\item $\text{Fun}^\times(\mathcal{M}(\text{Fin}_G),\mathcal{C})$ is a candidate for genuine $G$-objects of $\mathcal{C}$.
\end{enumerate}

\begin{example}
When $\mathcal{C}=\text{Top}$, $\mathcal{L}=\text{Fin}^\text{op}$ and the combinatorial construction $\mathcal{M}$ is the opposite category construction, so that (3) is Elmendorf's Theorem.

When $\mathcal{C}=\text{Sp}_{\geq 0}$, $\mathcal{L}=\text{Burn}$ and the combinatorial construction $\mathcal{M}$ is the virtual span construction, so that (3) is Guillou-May's Theorem.
\end{example}

\noindent One goal is to use this strategy to understand equivariant $\mathbb{E}_\infty$-ring spectra via the Lawvere theory of \emph{bispans} of finite $G$-sets, by analogy with the construction of Tambara functors \cite{Tambara}.

We hope to address these problems in a sequel, in which we will discuss combinatorial constructions of Lawvere theories (as in Principle \ref{PrCombLwv}).

\end{document}